\documentclass[11pt]{amsart}
 
\usepackage[euler-digits]{eulervm}

\usepackage{amsfonts, amsthm, amssymb, amsmath, stmaryrd}
\usepackage{mathrsfs,array}
\usepackage{eucal,fullpage,times,color,enumerate,accents}
\usepackage{url}

\usepackage{xcolor}
\usepackage{mathrsfs}
\usepackage{amssymb}
\usepackage{tikz}
\usepackage{tikz-cd}
\usepackage{bm}
\usepackage{enumerate}
\usetikzlibrary{calc}
\usepackage{subfigure,xypic}

\address{Department of Mathematics, Colorado State University}
\email{renzo@math.colostate.edu}

\address{School of Mathematics and Statistics, The University of Sheffield}
\email{paul.johnson@sheffield.ac.uk}

\address{Eberhard Karls Universit\"at T\"ubingen, Fachbereich Mathematik, Institut f\"ur Geometrie}
\email{hannah@math.uni-tuebingen.de}

\address{Department of Mathematics, Yale University}
\email{dhruv.ranganathan@yale.edu}

\usepackage[backref]{hyperref}
\hypersetup{
  colorlinks   = true,          
  urlcolor     = purple,          
  linkcolor    = blue,          
  citecolor   = purple             
}

\newtheorem{theorem}{Theorem}[subsection]

\newtheorem{lemma}[theorem]{Lemma}
\newtheorem{proposition}[theorem]{Proposition}
\newtheorem{definition}[theorem]{Definition}

\newtheorem{fact}{Fact}

\newtheorem{quasi-theorem}[theorem]{Quasi-Theorem}
\newtheorem{algorithm}[theorem]{Algorithm}

\newtheorem{rem1}[theorem]{Remark}
\newenvironment{remark}{\begin{rem1}\em}{\end{rem1}}

\newtheorem{ex1}[theorem]{Example}
\newenvironment{example}{\begin{ex1}\em}{\end{ex1}}

\newtheorem{not1}[theorem]{Notation}

\newcommand{\CC} {{\mathbb C}}          
            
\newcommand{\NN} {{\mathbb N}}		
\newcommand{\PP}{\mathbb{P}}         
\newcommand{\QQ} {{\mathbb Q}}		
\newcommand{\RR} {{\mathbb R}}		
\newcommand{\ZZ} {{\mathbb Z}}

\newcommand{\trop}{t\!r\!o\!p}

\newcommand{\mult}{m\!u\!l\!t}

\DeclareMathOperator{\Aut}{Aut}

\DeclareMathOperator{\val}{val}

\newcommand{\cal}{\mathcal}

\def\cZ{{\cal Z}}

\newcommand{\cutjoin}{\mathcal{F}_2}

\title{A graphical interface for the Gromov-Witten theory of curves}
\author{Renzo Cavalieri, Paul Johnson, Hannah Markwig, and Dhruv Ranganathan}
\date{\today}

\begin{document}

\pagestyle{plain}

\begin{abstract}
We explore the explicit relationship between the descendant Gromov--Witten theory of target curves, operators on Fock spaces, and tropical curve counting. We prove a classical/tropical correspondence theorem for descendant invariants and give an algorithm that establishes a tropical Gromov--Witten/Hurwitz equivalence. Tropical curve counting is related to an algebra of operators on the Fock space by means of bosonification. In this manner, tropical geometry provides a convenient ``graphical user interface'' for Okounkov and Pandharipande's celebrated GW/H correspondence. An important goal of this paper is to spell out the connections between these various perspectives for target dimension $1$, as a first step in studying the analogous relationship between logarithmic descendant theory, tropical curve counting, and Fock space formalisms in higher dimensions. 
\end{abstract}

\maketitle

\section{Introduction} The first goal of this article is to fill in the following square of correspondences:
\begin{equation}
\label{square}
\xymatrix{\mbox{\begin{tabular}{c}
Stationary descendant\\
Gromov-Witten invariants\\
of curves.
\end{tabular}} \ar@{-}[rr]^{\mbox{\cite{OP06}}} \ar@{--}[d]_{\textnormal{Thm.}\ \ref{thm-corres}}  & &  \mbox{Hurwitz numbers} \ar@{-}[d]^{\mbox{\cite{BBM10}}}\\
\mbox{\begin{tabular}{c}
Tropical stationary descendant\\
Gromov-Witten invariants\\
of curves.
\end{tabular}} \ar@{--}[rr]^{\textnormal{Alg.}\ \ref{tgwh} }  & & \mbox{Tropical Hurwitz numbers} 
}
\end{equation}
Hurwitz theory and Gromov-Witten (GW) theory are two enumerative geometric theories that count maps between Riemann Surfaces with specified discrete data. In~\cite{OP06}, Okounkov and Pandharipande develop the GW/H correspondence, a substitution rule that relates Hurwitz numbers and stationary descendant Gromov-Witten invariants of $\PP^1$. Tropical Hurwitz numbers are realized as weighted sums over appropriately decorated graphs:  they were first introduced in particular cases in \cite{CJM1,CJM2}; the general definition and correspondence theorem are due to Bertrand, Brugall\'e{} and Mikhalkin in \cite{BBM10}. 

In Definition \ref{tropgwi} we introduce tropical stationary descendant GW invariants of $\PP^1$ as weighted sums over decorated graphs, which realize tropical covers of the tropical model of the (twice punctured) projective line given by the real line. The key features are that an insertion of a stationary descendant class $\tau_k(pt)$ determines a relationship between the genus and the valence of a vertex of the contributing tropical covers. The multiplicities of vertices are relative GW invariants of $\PP^1$, relative to two points with a single descendant insertion. 
This definition extends naturally to an arbitrary target tropical curve (Definition \ref{tropgwigentc}), with the modification that certain vertices of the tropical covers have multiplicities given by local Hurwitz numbers.

With these definitions in place, we give direct proofs for the correspondences illustrated by the dotted lines above. In Theorem \ref{thm-corres} we establish that descendant GW invariants coincide with their tropical counterparts;  the tropical GW/H correspondence is proved by Algorithm \ref{tgwh}, obtained by globalizing a simple local surgery on graphs. In both cases, the proof is a simple consequence of the degeneration formula (Theorem  \ref{degfor}), which roughly speaking means that we can ``break" the curves we are interested in counting into nodal curves, where each irreducible component becomes simpler; then tropical covers keep track of how these components are assembled. 

The fact that tropical geometry is an effective tool to organize the combinatorics of degenerations is not just a lucky coincidence. 
 In the present context, the relevant degenerations are relative stable maps to target curves. Given a one parameter proper, semistable degeneration of a curve, the dual graph of the special fiber can be endowed with a natural metric. Degenerations of relative stable maps produce piecewise linear maps between these metric graphs, which are tropical covers. The expansion factors along edges of these maps contribute multiplicities to the tropical count in a natural way. The multiplicity reflects the number of algebro-geometric objects tropicalizing to a fixed tropical object. Note that in other instances of correspondence theorems in tropical geometry, such as double Hurwitz numbers, toric Severi degrees, and rational curve counts in toric varieties, the multiplicity is purely combinatorial, see~\cite{CJM1,Mi03,NS06,R15b,Tyo09}. In geometrically richer situations, it is necessary to use certain enumerative invariants as ``local input data''. This is the case for the general version of the \emph{Hurwitz problem} counting covers of a fixed curve of fixed genus and degree and prescribed ramification data over fixed branch points~\cite{BBM10,CMR14}, and for our tropical descendant GW invariants.

As stated earlier, the GW/H correspondence of \cite{OP06} is expressed as a substitution rule: each insertion of the form $\tau_k (pt)$ for a GW invariant becomes a specific linear combination of ramification conditions for the Hurwitz problem. Thinking of Hurwitz numbers as multilinear functions on the space of ramification conditions then allows us to express a given descendant GW invariant in terms of a (rather large) linear combination of Hurwitz numbers. The degeneration formula alone shows that there should be a substitution rule between GW invariants and Hurwitz numbers. In order to explicitly describe the substitution rule, relative descendant GW invariants are expressed as matrix elements on fermionic Fock space. Here we show that relative descendant GW invariants are matrix elements on bosonic Fock space: in  oversimplifying terms, any descendant insertion $\tau_k(pt)$ corresponds to a linear operator  $M_k$ on a countably dimensional vector space with a distinguished basis which indexes relative conditions; then a  descendant GW invariant with specified relative conditions at two points is a structure coefficient (determined by the relative conditions) of a composition (determined by the descendant insertions) of such linear operators. In the mathematical physics community it is well known that matrix elements on bosonic Fock space can be computed by Wick's Theorem, via a Feynmann diagrams expansion. It is probably not too surprising at this point that the Feynmann graphs for descendant GW invariants are precisely the tropical covers that contribute to the corresponding tropical descendant GW invariants.

At this point hopefully we have revealed the second goal of this paper, which is to draw attention to the connections between tropical geometry, degenerations, and  operator formalism on Fock space, and make such connections explicit and approachable to a larger group of mathematicians. Interestingly, while such connections are most transparent when the target dimension is one, it is recent work on the Gromov-Witten theory of surfaces  that inspired us to revisit the target dimension one case: in \cite{CP}, the Severi degrees of $\PP^1\times\PP^1$ and $\PP^2$ are expressed as matrix elements of the exponential of a single operator on bosonic Fock space; in \cite{BG14}, tropical counts and floor diagrams are used to study primary relative GW invariants of certain families of toric surfaces, and again these invariants are interpreted as matrix elements for a particular operator. In forthcoming work, we intend to extend the reach of these result by incorporating descendant insertions, and further exploring the structure of GW invariants on toric surfaces. 

The paper is organized as follows. Section \ref{prelims} is dedicated to the part of Diagram \ref{square} that existed before this paper: we recall classical and tropical Hurwitz numbers, classical stationary descendant GW invariants of curves, genus zero tropical descendant GW invariants and give a brief account of the GW/H correspondence. In Section \ref{sec: desc-corr} we introduce tropical, stationary descendant GW invariants and prove the correspondence Theorem \ref{thm-corres}. We conclude Section \ref{sec: desc-corr} with an explicit evaluation of the genus $0$ and genus $1$ vertex multiplicities in the definition of tropical GW invariants. In Section  \ref{sec-gwh} we describe the algorithm that establishes the tropical GW/H correspondence. In Section \ref{fock} we introduce the bosonic Fock space and express descendant GW invariants of $\PP^1$ as matrix elements, showing in the process the correspondence between tropical cover and Feynmann diagrams.

\subsection{Acknowledgements}
Part of this work was completed during a \textit{Research in Pairs} stay at the Oberwolfach Institute for Mathematics. We thank the institute for its hospitality and for providing excellent working conditions. RC was supported from  NSF grant DMS-1101549, NSF RTG grant 1159964. HM acknowledges support by DFG SFB/TRR 195. DR was supported by NSF grant CAREER DMS-1149054 (PI: Sam Payne). We thank the referee for comments that improved the quality of the exposition.

\section{Preliminaries}\label{prelims}

\subsection{Hurwitz theory} 

In this section we recall the definition of Hurwitz numbers, introduce tropical curves and their covers, and state the classical/tropical Hurwitz correspondence theorem.

\subsubsection{Classical Hurwitz theory}

Classically, Hurwitz numbers count the number of maps that exist between Riemann surfaces with prescribed discrete invariants.

\begin{definition}
Let $Y$ be a smooth Riemann surface of genus $h$. Fix integers $d$ and $g$, and $n$ points $p_1,\ldots,p_n$ in $D$ along with partitions $\mu_1,\ldots,\mu_n$ of $d$. The \emph{Hurwitz number} $H_{g\rightarrow h}(\mu_1,\ldots,\mu_n)$
is the weighted count of covers $f:X\rightarrow Y$, where:
\begin{itemize}
\item $X$ is a connected smooth Riemann surface of genus $g$.
\item The branch points of $f$ are exactly the points $p_1,\ldots,p_n$, and the ramification profile over $p_i$ equals $\mu_i$.
\item Each cover is weighted by $\frac{1}{|\Aut(f)|}$.
\end{itemize}
\end{definition} 

Hurwitz numbers are topological invariants, they do not depend on the choice of the complex structure on $Y$ nor on the positions of the $p_i$. 
Denote the parts of a partition $\mu_i$ by $\mu_{ij}$. The Riemann--Hurwitz formula
\begin{equation} 2-2g=d(2-2h)-\sum_{i=1}^n \sum_{j=1}^{\ell(\mu_i)} (\mu_{ij}-1) .\label{eq-RH}\end{equation}
is a necessary but not sufficient condition for the existence of a cover with the specified ramification data. 

There is essentially a bijection between maps of Riemann surfaces and equivalence classes of monodromy representations, up to postcomposition with inner automorphisms of the symmetric group. This allows to translate the geometric count of maps of Riemann surfaces into a group theoretic count.

A \textit{monodromy representation} of type $(g\to h,d,\mu_1,\ldots, \mu_n)$ is a group homomorphism $\Phi:\pi_1(Y\smallsetminus \{p_1,\ldots,\mu_n\},y_0) \to S_d$ such that, if $\rho_k$ is the homotopy class of a small loop around $p_k$, then the permutation $\Phi(\rho_k)$ has cycle type $\mu_k$. We call the monodromy representation \textit{connected} if the subgroup $Im(\Phi) \leq S_d$ acts transitively on the set $\{1,2,\ldots,d\}$. 

\begin{fact}
$H_{g\rightarrow h}(\mu_1,\ldots,\mu_n)$ is the number of connected monodromy representations of type $(g\to h,d,\mu_1,\ldots, \mu_n)$ divided by $d!=|S_d|$.
\end{fact} 

Counting all mondoromy representations rather than just the connected ones counts covers from possibly disconnected source curves. Such count is called {\it disconnected Hurwitz number} and denoted $H^\bullet_{g\rightarrow h}(\mu_1,\ldots,\mu_n)$. Disconnected Hurwitz numbers are the natural setting for yet another translation of our enumerative geometric problem. The set of  (not necessarily connected) monodromy representations of type $(g\to h,d,\mu_1,\ldots, \mu_n)$ is in a natural way the solution to a multiplication problem in the class algebra $\cZ\CC[S_d]$ of the symmetric group.

For $\mu$ a partition of $d$, let $C_{\mu}$ denote the formal sum of all permutations of $S_d$ of cycle type $\mu$. The collection of $\{C_\mu\}_{\mu\vdash d}$ is a basis for the class algebra of the symmetric group, called the {\it conjugacy class} basis. For a fixed positive integer $d$, define
\begin{equation}
\mathfrak{K}:= \sum_{\mu \vdash d} |\mathfrak(\mu)| C_\mu^2 \in \cZ\CC[S_d],
\end{equation}
where $|\mathfrak(\mu)|$ is the size of the centralizer of any permutation in the conjugacy class indexed by $\mu$. Note that $\cZ\CC[S_d]$ is the class algebra of the symmetric group, i.e. the complex center of the group ring of $S_d$.

\begin{fact}
The number $d! H^\bullet_{g\rightarrow h}(\mu_1,\ldots,\mu_n)$  is the coefficient of $C_e$ in the product $\mathfrak{K}^h C_{\mu_1}\ldots C_{\mu_n}$, expanded in the conjugacy class basis of $\cZ\CC[S_d]$.
\end{fact} 

The class algebra of the symmetric group is semisimple. A natural semisimple basis is indexed by irreducible representations and the coefficients of the change of basis from the conjugacy class basis are the characters of $S_d$. This allows one to write closed formulas for Hurwitz numbers in terms of character theory of the symmetric group. In this text we will not work closely with such formulas, but the connection with character theory of the symmetric group is at the heart of the natural interplay between Hurwitz theory and the algebra of operators on fermionic Fock space. See Section~\ref{fock}.

 A very basic reference for Hurwitz numbers from an enumerative geometric point of view is \cite{renzosbook}.

\subsubsection{Tropical Hurwitz theory}
We now introduce basic concepts in tropical geometry that we need to define tropical Hurwitz numbers.

An \emph{abstract tropical curve} is a connected metric graph $\Gamma$, such that edges leading to leaves (called \emph{ends}) have infinite length, together with a genus function $g:\Gamma\rightarrow \ZZ_{\geq 0}$ with finite support. Locally around a point $p$, $\Gamma$ is homeomorphic to a star with $r$ halfrays. 
The number $r$ is called the \emph{valence} of the point $p$ and denoted by $\val(p)$. We identify the vertex set of $\Gamma$ as the points where the genus function is nonzero, together with points of valence different from $2$. The vertices of valence greater than $1$ are called  \textit{inner vertices}. Besides \emph{edges}, we introduce the notion of \emph{flags} of $\Gamma$. A flag is a pair $(V,e)$ of a vertex $V$ and an edge $e$ incident to it ($V\in \partial e$). Edges that are not ends are required to have finite length and are referred to as \emph{bounded} or \textit{internal} edges.

A \emph{marked tropical curve} is a tropical curve whose leaves are (partially) labeled. An isomorphism of a tropical curve is an isometry respecting the leaf markings and the genus function. The \emph{genus} of a tropical curve $\Gamma$ is given 
\[
g(\Gamma) = h_1(\Gamma)+\sum_{p\in \Gamma} g(p)
\]
A curve of genus $0$ is called \emph{rational} and a curve satisfying $g(v)=0$ for all $v$ is called \emph{explicit}. The \emph{combinatorial type} is the equivalence class of tropical curves obtained by identifying any two tropical curves which differ only by edge lengths.

\begin{definition}
\label{def:tropcov}
A \emph{tropical cover} $\pi:\Gamma_1\rightarrow \Gamma_2$ is a surjective harmonic map of metric graphs in the sense of~\cite[Section 2]{ABBR1}. The map $\pi$ is piecewise integer affine linear, the slope of $\pi$ on a flag or edge $e$ is a positive integer called the \emph{expansion factor} $\omega(e)\in \NN_{> 0}$. 
\end{definition}

For a point $v\in \Gamma_1$, the \emph{local degree of $\pi$ at $v$} is defined as follows. Choose a flag $f'$ adjacent to $\pi(v)$, and add the expansion factors of all flags $f$ adjacent to $v$ that map to $f'$:
\begin{equation}
d_v=\sum_{f\mapsto f'} \omega(f).
\end{equation} 
We define the \textit{harmonicity} or \textit{balancing condition} to be the fact that for each point $v\in \Gamma_1$, the local degree at $v$ is well defined (i.e. independent of the choice of $f'$).

A tropical cover is called a \emph{tropical Hurwitz cover} if it satisfies the local Riemann-Hurwitz condition at each point \cite{BBM10, Cap12, CMR14} stating that when $v\mapsto v'$ with local degree $d$ 
\begin{equation}
2-2g(v)= d(2-2g(v'))-\sum (\omega(e)-1),
\end{equation}
 where the sum goes over all flags $e$ adjacent to $v$. 

The \emph{degree} of a tropical cover is the sum over all local degrees of preimages of a point $a$, $d=\sum_{p\mapsto a} d_p$. By the balancing condition, this definition does not depend on the choice of $a\in \Gamma_2$. For an end $e\in \Gamma_2$ let $\mu_e$ be the partition of expansion factors of the ends of $\Gamma_1$ mapping onto $e$. We call $\mu_e$ the \emph{ramification profile} above $e$.

For a vertex $v\in \Gamma_1$, define a local Hurwitz number 
\begin{equation}\label{lochu}H(v):=H_{g\rightarrow h}(\mu_1,\ldots,\mu_n)\cdot \prod_{i=1}^n|\Aut(\mu_i)|,\end{equation} where $g$ is the genus at $v$, $h$ is the genus at the image of $v$, and the $\mu_i$ are defined as the partitions of expansion factors of inverse images for each flag incident to  the image of $v$. The group $\Aut(\mu_i)$ is defined as follows. Write the partition as $\mu_i = \alpha_1^{k_1}\alpha_2^{k_2}\cdots\alpha_r^{k_r}$, with the $\alpha_j$ distinct integers, indicating that $\alpha_j$ appears $k_j$ times. Then, 
\[
\Aut(\mu_i) := \prod_{j=1}^r S_{k_j},
\]
and thus $|\Aut(\mu_i)| = \prod_j k_j!$.

\begin{definition}\label{def-tropHurwitz}
Let $\Gamma_2$ be an explicit tropical curve of genus $h$ with $n$ leaves. Fix integers $d$ and $g$, and $r$ partitions $\mu_1,\ldots,\mu_n$ of $d$ satisfying Equation (\ref{eq-RH}).
The \emph{tropical Hurwitz number} $H_{g\rightarrow h}^{\trop}(\mu_1,\ldots,\mu_n)$
is the weighted count of tropical Hurwitz covers $\pi:\Gamma_1\rightarrow \Gamma_2$, where
\begin{itemize}
\item $\Gamma_1$ is a tropical curve of genus $g$.
\item The ramification profiles over the leaf $i$ of $\Gamma_2$ equals $\mu_i$.
\end{itemize}
Each  cover is weighted by \begin{equation}\mult_{H}(\pi):=\frac{1}{|\Aut(\pi)|}\cdot \prod_e\omega(e)\cdot \prod_v H(v).\label{eq-multHurwitz}\end{equation} Here, the first product goes over all edges $e$ of $\Gamma_1$ and $\omega(e)$ denotes the expansion factor. The second product goes over all vertices of $\Gamma_1$ and $H(v)$ denotes their local Hurwitz number.
\end{definition} 

That the above definition depends only on the genus of $\Gamma_2$ and not the precise graph can be seen as a consequence of the correspondence theorem. 

\begin{theorem}[Correspondence theorem for Hurwitz numbers]
Algebraic Hurwitz numbers equal their tropical counterparts, i.e.\ we have $$ H_{g\rightarrow h}(\mu_1,\ldots,\mu_n)= H_{g\rightarrow h}^{\trop}(\mu_1,\ldots,\mu_n).$$

\end{theorem}

We refer the reader to~\cite{BBM10,CMR14} for a proof.

If we pick a target having only $3$-valent inner vertices the local Hurwitz numbers we need to take into account for the computation of the multiplicity of a tropical Hurwitz covers are all triple Hurwitz numbers.
In this sense, we can view triple Hurwitz numbers as local input data we make use of to compute any Hurwitz number with the tropical approach.

\subsection{Gromov-Witten theory of \texorpdfstring{$\PP^1$}{P1}}

\subsubsection{Algebraic Gromov-Witten theory of \texorpdfstring{$\PP^1$}{P1}}

Gromov- Witten invariants are virtually enumerative intersection numbers on moduli spaces of stable maps.

A  {\it stable map} of degree $d$ from a curve of genus $h$  to $\PP^1$  is a function $f: X \to \PP^1$, where  $X$ is a connected projective curve with at worst nodal singularities, such that $f_\ast([X]) = d[\PP^1]$ and $f$ has a finite group of automorphism. 

One may endow $X$ with the extra structure of a set of markings, i.e. $x_1, \ldots, x_n$  distinct nonsingular points of $X$. In concrete terms, the condition on a finite group of automorphisms requires at least three special points, i.e. marks or nodes, on any rational component contracted by $f$.

The moduli space of stable maps, denoted  $\overline{\mathcal{M}}_{g,n} (\PP^1,d)$ is a proper Deligne-Mumford stack of virtual dimension $2g-2+2d+n$. When  $g>0$ this space typically contain many components that exceed the virtual dimension. For example, $\overline{\mathcal{M}}_{g} (\PP^1,1)$ has virtual dimension $2g$, but it has a set of components parameterized by partitions of $g$, ranging from dimension $3g-1$  to dimension $2g$.  The {\it virtual fundamental class} $[\overline{\mathcal{M}}_{g,n} (\PP^1,d)]^{vir}$ is a Chow class
of pure dimension $2g-2+2d+n$ that shares many formal properties with the fundamental class of a smooth manifold \cite{Beh97,BF97}. 

To define enumerative invariants that have similar flavor to Hurwitz numbers, we must define some zero dimensional cycles on the moduli space of stable maps by intersecting cycles which have a similar meaning to fixing branch points and prescribing ramification conditions. This is achieved, by  using {\it evaluation morphisms} and {\it descendant insertions}.

The $i^{\textnormal{th}}$  evaluation morphism is the map $ev_i: \overline{\mathcal{M}}_{g,n} (\PP^1,d) \to \PP^1$ that sends a point $[X, x_1, \ldots, x_n, f]$ to  $f(x_i)
\in \PP^1$. 

The $i^{\textnormal{th}}$ cotangent line bundle $\mathbb{L}_i \to \overline{\mathcal{M}}_{g,n} (\PP^1,d)$ is defined by a canonical identification of its fiber over a moduli point $(X, x_1, \ldots, x_n, f)$
with the cotangent space $T^\ast_{x_i}(X)$.  The first Chern class of the cotangent line bundle is called a {\it psi class} ($\psi_i= c_1(\mathbb{L}_i)$), and any monomial in psi classes is called a {\it descendant class}.

 \begin{definition}
 Fix $g,n,d$ and let $k_1, \ldots, k_n$ be non-negative integers with $k_1+\ldots+k_n = 2g+2d-2$. The {\it stationary Gromov-Witten invariant}  $\langle \tau_{k_1}(pt) \ldots \tau_{k_n}(pt) \rangle_{g,n}^{\PP^1,d}$ is  defined by:
 \begin{equation}
 \langle \tau_{k_1}(pt) \ldots \tau_{k_n}(pt) \rangle_{g,n}^{\PP^1,d} = \int_{[\overline{\mathcal{M}}_{g,n} (\PP^1,d)]^{vir}} \prod_{i=1}^n ev_i^\ast(pt) \psi_i^{k_i},
 \end{equation}
 where $pt$ denotes the class of a point in $\PP^1$.
 \end{definition}

As before, one can allow source curves to be disconnected, and introduce {\it disconnected Gromov-Witten invariants}. We will add the superscript $\bullet$ anytime we wish to refer to the disconnected theory.

We conclude this section by introducing some invariants that are in a sense a hybrid between Hurwitz numbers and Gromov-Witten invariants. They are constructed using moduli spaces of {\it relative stable maps} $\overline{\mathcal{M}}_{g,n} (\PP^1, \mu_1,\ldots,\mu_n,d)$, where part of the data specified is the ramification profile $\mu_1, \ldots, \mu_n$  for the map over $r$ labeled points on the base curve. For such a moduli space to be proper, the target $\PP^1$  may sprout rational tails. A detailed discussion of the boundary of spaces of relative stable maps to $\PP^1$ is not necessary for our current scope, but the interested reader may find it in \cite{Vak08}. When we refer to relative Gromov-Witten invariants with at most two relative insertions, we will use operator notation; if $\mu_1, \mu_2$ are two partitions of $d$ denoting ramification profiles over $0$ and $\infty \in \PP^1$:
\begin{equation}
 \langle \mu_1| \tau_{k_1}(pt) \ldots \tau_{k_n}(pt)|\mu_2 \rangle_{g, n}^{\PP^1,d} = \int_{[\overline{\mathcal{M}}_{g,n} (\PP^1, \mu_1,\mu_2,d)]^{vir}} \prod_{i=1}^n ev_i^\ast(pt) \psi_i^{k_i}
\end{equation}  

\subsubsection{Tropical Gromov-Witten theory of $\PP^1$}

Ultimately, one wishes to define {\it tropical Gromov-Witten invariants} as intersection numbers on moduli spaces of stable maps from tropical curves to tropical $\PP^1$. 
The difference between  a tropical (relative) stable map  and  a tropical cover is essentially only that one may allow edges of the source graph to be contracted to vertices of the target graph.

\begin{definition}
\label{def:troprsm}
Let $\Gamma_2$ be a tropical curve with $r$ leaves, and $\mu_1, \ldots, \mu_n$ partitions of the positive integer $d$.
A \emph{tropical relative stable map}, with relative condition prescribed by the  partitions $\mu_i$, is  a map $\pi:\Gamma_1\rightarrow \Gamma_2$ which is a surjective harmonic map of metric graphs; the map $\pi$ is piecewise integer affine linear, with nonnegative integer slope on each edge: $\omega(e)\in \NN_{\geq 0}$. Further,  the inverse image of the $i^{\textnormal{th}}$ leaf of $\Gamma_2$ consists of $\ell(\mu_i)$ leaves of $\Gamma_1$  weighted by the parts of $\mu_i$.
\end{definition}

Parameter spaces for relative stable maps with fixed discrete invariants are in a natural way generalized cone complexes. Just like algebraic moduli spaces of stable maps, when the genus of the source graph is positive, the maximal dimensional cones have different dimensions, which exceed  the dimension of the cones parameterizing honest covers. One wishes to cut out the right dimensional components from these excess dimensional cones, that is, develop a ``tropical virtual fundamental class". The latter has not yet been established, and so at this point tropical Gromov-Witten invariants may not be constructed as degrees of $0$-dimensional intersection cycles. However, the moduli space of maps from genus $0$ curves to $\PP^1$ relative to two points is smooth. The natural boundary divisor, parametrizing maps from curves with at least one node, has simple normal crossings. In this case, the cone complex parametrizing tropical maps is the cone over the dual complex of the boundary. Furthermore, it has the structure of a fan, which allows one to define psi classes intrinsically in tropical geometry. These classes are the tropicalization of generic representatives of the algebraic classes. The moduli space of relative stable maps embeds naturally in a toric variety, and techniques from toric intersection theory allow one to then ``see'' why tropical computations are carrying out intersection theory on the algebraic moduli space, to calculate degrees of the intersections of cycles of the form $ev_i^{-1}(pt)\psi_i^{k_i}$. The details may be found in~\cite{CMR14b}. The main consequence is the following. 

\begin{theorem}[Correspondence theorem for rational descendant invariants]
\label{thm:descendant-correspondence}
We have an equality of classical and tropical genus $0$ relative descendant Gromov--Witten invariants of $\PP^1$:
\[
\langle \tau_{k_1}(pt),\ldots, \tau_{k_n}(pt)\rangle^{\PP^1,d}_{0,n}=
\langle \tau_{k_1}(pt),\ldots, \tau_{k_n}(pt)\rangle_{0,n}^{\PP^1,d,\trop}.
\]
\end{theorem}

This is in a sense the most satisfying type of correspondence theorem, in the sense that the correspondence of the enumerative geometric information is in fact a consequence of a deeper correspondence in the geometry and intersection theory of the algebraic and combinatorial moduli spaces involved. A similar ``geometrized'' correspondence theorem for Hurwitz numbers is proved in~\cite{CMR14}. In this setting, tropical enumerative computations are related to the calculation of a degree of the appropriate evaluation morphism on analytic domains in a Berkovich analytic space. 

The lack of understanding of the correspondence at the level of moduli spaces does not prevent one from studying correspondence between enumerative geometric quantities, and in fact such numerical correspondences are  evidence that it makes sense to look for a geometric correspondence between the moduli spaces. In Section~\ref{sec: desc-corr}, we bypass the virtual class issue and define stationary descendant invariants of $\PP^1$ as an appropriately weighted sum of graphs with certain characteristics, and then prove a correspondence theorem with the classical invariants.

\subsection{The GW/H correspondence} 

Let $V = \oplus_{k=0}\QQ \tau_k(pt)$ be a countably dimensional  rational vector space with a basis given by the symbols $\tau_k(pt)$, and  $W = \oplus_{d=0}^\infty \oplus_{\mu\vdash d} \QQ \mu$ a vector space with a basis given by all partitions of all nonnegative integers. Here, $\emptyset$ is considered the unique partition of $0$. The collection of all $n$-pointed stationary Gromov-Witten invariants  of $\PP^1$ defines a multilinear function  $GW^\bullet_n: V^{\otimes n}\to \QQ$, defined on the elements of the natural basis of $V^{\otimes n}$ by:
\begin{equation}
\tau_{k_1}(pt) \otimes \ldots \otimes \tau_{k_n}(pt)\mapsto \sum_{d}\langle \tau_{k_1}(pt) \ldots \tau_{k_n}(pt) \rangle_{n}^{\PP^1,d, \bullet} q^d,
\end{equation}
where $q$ is a formal variable keeping track of degree and we omit the genus subscript since the genus of the source curve is  determined by the other discrete invariants.

Similarly, Hurwitz theory for target curve $\PP^1$ with $n$ branch points defines a multi-linear function $H^\bullet_n: W^{\otimes n} \to \QQ$: 
\begin{equation} \label{hurwmlf}
\mu_1 \otimes \ldots \otimes \mu_n \mapsto \sum_d H^\bullet_d(\mu_1, \ldots, \mu_n) q^d,
\end{equation}
where we omit the genera subscripts since the genus of the base curve is fixed and the genus of the source curve is determined via the Riemann-Hurwitz formula. An important detail to pay attention to is that  the partitions $\mu_i$ are not necessarily partitions of $d$, so we adopt the following conventions in order for \eqref{hurwmlf} to make sense:
\begin{itemize}
	\item $H^\bullet_0(\emptyset, \ldots, \emptyset):=1$.
	\item if some of the $\mu_i$ are partitions of an integer strictly greater than $d$, then we set $H^\bullet_d(\mu_1, \ldots, \mu_n):=0$.
	\item if $\mu$ is a partition of $d-k$, for $k\geq 0$, then we define $\tilde\mu$ to be the partition of $d$ consisting of the parts of $\mu$ plus $k$ parts of size $1$. Moreover let  $j$ denote the number of parts of $\mu$ of size $1$. Then we define:
	\begin{equation} \label{mutilde}
	H^\bullet_d(\mu_1, \ldots, \mu_n):= \prod_{i=1}^n {{j_i +k_i}\choose{k_i}}H^\bullet_d(\tilde\mu_1, \ldots, \tilde\mu_n)
	\end{equation}
\end{itemize}
While definition \eqref{mutilde} may seem unnatural at first, it has an intuitive geometric interpretation: we are counting Hurwitz covers corresponding to the branching data $\tilde\mu_i$'s together with a choice, for each branch point $p_i$, of a divisor of unramified preimages of $p_i$ of size $k_i$.

The GW/H correspondence of \cite{OP06} may now be stated as the existence of a linear transformation $CC: V \to W$ that gives rise, for every $n$, to the commutative diagram:
$$
\xymatrix{V^{\otimes n} \ar[dr]_{GW_n} \ar[rr]^{CC^{\otimes n}} & & \ar[dl]^{H_n} W^{\otimes n}\\
  & \QQ & }
$$

Defining the linear transformation $CC$ in detail requires a fair amount of work and a detour into the connection between character theory of the symmetric group and the theory of shifted symmetric functions. We refer the reader to the elegant treatment in~\cite[Section 0.4]{OP06}) Here we point out some of the features of this transformation.
For every $k\geq 0$, $CC$ is defined by 
\begin{equation}
\tau_k(pt) \mapsto \frac{1}{k!} \overline{(k+1)},
\end{equation}
where $\overline{(k+1)}= (k+1)+ \sum \rho_{k+1,\mu} \mu$ is  called a {\it completed cycle}, and it has the following characteristics:
\begin{itemize}
      \item $\sum \rho_{k+1,\mu} \mu$ is a linear combination of partitions of integers strictly less than $k+1$.
      \item the {\it completion coefficients}  $\rho_{k+1, \mu}$ with $\mu \not= \emptyset$ are non-negative rational numbers.
      \item for every integer $g\geq 0$, the {\it completion coefficient}  $\rho_{2g-1, \emptyset}$ is equal to the Hodge integral $(-1)^g\int_{\overline{M}_{g,1}} \lambda_g \psi^{2g-2}$.
\end{itemize}
If the reader is unfamiliar with Hodge integrals, then the content that should be taken away from the last bullet point is that completion coefficients in degree $0$ have some natural geometric meaning. This is true in fact of all completion coefficients, which are essentially  connected relative $1$-point
Gromov-Witten invariants of $\PP^1$ relative to $0 \in \PP^1$.

\begin{proposition}[{\cite[Proposition 1.6]{OP06}}]
\label{compco}
Write $\mu = \alpha_1^{k_1}\cdots \alpha_r^{k_r}$. The completion coefficients are given by the formula
\begin{equation}
\frac{\rho_{k+1, \mu}}{k!} = \prod_j k_j! \alpha_j^{k_j}\cdot \langle \mu | \tau_k(pt) \rangle^{\PP^1,d}_{n=1},\label{eq-compco}
\end{equation}
where $\mu \vdash d$ and the genus of the source curve is determined by the other discrete invariants.
\end{proposition}

The GW/H correspondence, and the geometric interpretation of the completion coefficients arise as  a consequence of the {\it degeneration formula} \cite{Li01A,Li02}, which  states that a Gromov-Witten invariant for a target curve $X$ can be computed by degenerating $X$ to a nodal curve $X_1\cup_x X_2$ and appropriately combining invariants of $X_1$ and $X_2$ relative to the point $x$.
 More precisely,
 \begin{theorem}
 \label{degfor}
 Let ${\mathscr X}_t$ be a flat family of curves such that the general fiber is a smooth curve $X$ and the central fiber is a nodal curve $X_1\cup_x X_2$, $s_1(t), \ldots, s_{n+m}(t)$ pairwise disjoint sections of ${\mathscr X}_t$ with $s_1(0), \ldots , s_n(0) \in X_1$ and $s_{n+1}(0), \ldots , s_{n+m}(0) \in X_2$. Then
%
 \begin{align}\begin{split}
& \langle \tau_{k_1}(pt),\ldots, \tau_{k_{n+m}}(pt)\rangle^{X, d \bullet}_{n+m} \\= & \sum_{\mu\vdash d} |\xi(\mu)|  \langle \tau_{k_1}(pt),\ldots, \tau_{k_{n}}(pt)\rangle^{X_1, d,\bullet}_{n}  \langle \tau_{k_{n+1}}(pt),\ldots, \tau_{k_{n+m}}(pt)\rangle^{X_2, d, \bullet}_{m}.\end{split}
 \end{align}

 \end{theorem}

Theorem \ref{degfor} can be applied to the case where one imposes relative conditions on the target curves, to compute relative Gromov-Witten invariants. It also generalizes in a natural way to the situation where $X$ degenerates to a nodal curves with more than $2$ components.
The $GW/H$ correspondence is obtained by degenerating $(\PP^1, p_1, \ldots, p_n)$ to a ``comb" of rational curves  with each of the $p_i$ on one of the teeth of the comb. Then the invariants corresponding to the spine of the comb have only relative insertions, and give therefore Hurwitz numbers. For each tooth of the comb, a simple dimension count shows that the disconnected, descendant, one pointed invariant with one relative condition $\tilde\mu \vdash d$ must actually consist of a connected, descendant, one pointed invariant with one relative condition $\mu \vdash d-k$, where the parts of $\tilde\mu$ consists of the parts of $\mu$ plus $k$ parts of size one.

While the degeneration formula explains in a very satisfactory way the structure of the GW/H correspondence, and the meaning of the completion coefficients, the connection to the theory of shifted symmetric function is what allowed Okounkov and Pandharipande to explicitly evaluate the completion coefficients. We review this part of the story, and the connections to matrix elements in the Fock space, in Section \ref{fock}.

\section{Correspondence theorem for tropical descendant GWI}\label{sec: desc-corr}

In this section we define tropical descendant GW invariants as an appropriate weighted  graph sum, and prove a correspondence  theorem that equates them with classical descendant invariants. We first study the case when the target is $\PP^1$, where the notation can be reduced to a minimum. The case of a general target curve is not conceptually more difficult and will be sketched later in the section.

\subsection{Target curve \texorpdfstring{$\PP^1$}{P1}} 

\begin{definition}\label{tropgwi}
Let $g$ be a non-negative integer, $\mu, \nu$ two partitions of a positive integer $d$; let $k_1, \ldots, k_n$ be nonnegative integers such that $\sum k_i = 2g+\ell(\mu)+\ell(\nu)-2$. We define
the {\it tropical descendant GW invariant}
$$\langle \mu| \tau_{k_1}(pt)\ldots\tau_{k_n}(pt)|\nu \rangle ^{\PP^1, d, trop}_{g,n}
= \sum_{\pi:\Gamma\to \PP^1_{\trop}} \frac{1}{| \Aut(\pi)|} \prod_V m_V \prod_e \omega(e)
$$
where
\begin{enumerate}[(A)]
\item The map $\pi:\Gamma\to \PP^1_{\trop}= \RR$ is a connected tropical cover.
\item The unbounded left (resp. right ) pointing ends of $\Gamma$ have expansion factors given by the partition $\mu$ (resp. $\nu$).
\item There is a unique vertex of $\Gamma$ over each one of $n$ fixed points $p_1,\ldots, p_n$ in $\PP^1_{\trop}$.
\item The unique vertex $v_i$ over $p_i$  has genus $g_i$, and valence $k_i+2-2g_i$.  If the star of $v_i$ has right (resp. left) hand side expansion factors given by $\mu_i$ (resp. $\nu_i$), the multiplicity of the vertex $v_i$ is defined to be
\begin{equation}\label{localmv}
m_{v_i} = |Aut(\mu_i)||Aut(\nu_i)|\int_{\overline M_{g_i,1}(\PP^1,\mu_i,\nu_i)} \psi^{k_i} ev^\star(pt).
\end{equation}
\item The product of the expansion factors $\omega(e)$ runs over the set of all bounded edges of $\Gamma$.
\end{enumerate}
\end{definition}

One may define {\it disconnected} tropical descendant invariants simply by removing the requirement that the graph $\Gamma$ be connected. Note however that the vertex multiplicities $m_{v_i}$ remain connected one-pointed descendant invariants.  We now state the correspondence theorem --- note that an analogous statement holds in the disconnected case.

\begin{theorem}\label{thm-corres} For any choice of discrete data, tropical and classical descendant invariants agree:
$$\langle \mu| \tau_{k_1}(pt)\ldots\tau_{k_n}(pt)|\nu \rangle ^{\PP^1,d}_{g,n}
= \langle \mu| \tau_{k_1}(pt)\ldots\tau_{k_n}(pt)|\nu \rangle ^{\PP^1, d, \trop}_{g,n}$$
\end{theorem}

\begin{proof} This follows from a application of the degeneration formula in Gromov--Witten theory (Theorem \ref{degfor}). We indicate the structure of this argument and omit the standard bookkeeping.

Consider a degeneration of the base curve $\PP^1$ into a chain  $C$ of $n$ rational curves, such that the point $p_i$ is on the $i^{\textnormal{th}}$ component. By the degeneration formula, the classical descendant invariant  $\langle \mu| \tau_{k_1}(pt)\ldots\tau_{k_n}(pt)|\nu \rangle ^{\PP^1,d}_{g,n}$ may be computed in terms of relative stable maps to $C$, as a sum of contributions from all possible ways of imposing relative conditions at each of the nodes of $C$. For each choice that gives a non-zero contribution, one may consider the dual graph $\Gamma$ of a corresponding map to $f:X \to C$.  Give each edge $e$, corresponding to a note of $X$ mapping to a node of $C$, expansion factor equal to the order of ramification of the map at the node.\footnote{More precisely, we should speak about the order of ramification of the normalization of $f$ at either of the shadows of the node of $C$. Such order is well defined by the {\it kissing condition} at nodes over relative points.}  The key point is that  after forgetting two valent genus zero vertices there is a bijection between such dual graphs and the tropical covers in Definition \ref{tropgwi}: by dimension reasons, over each component of $C$, all connected components that do not host the marked point must be rational curves mapping to the base trivially (i.e. ramifying only over the nodes). Also, dimension reasons imply that the discrete invariants for the map restricted to the component containing the marked point must satisfy \eqref{localmv}. The proof is concluded by noting that the degeneration formula multiplicity  for a given type of relative map equals the multiplicity of the corresponding tropical cover.
\end{proof}

The proof of Theorem \ref{thm-corres} is purely combinatorial, and it may fall short of explaining why tropical covers encode so perfectly the  combinatorics of the degeneration formula. 
At this point we do understand that this is not just a lucky coincidence, but a consequence of a deeper correspondence between algebraic relative stable maps and their tropical counterparts. For the interested readers, we expand upon this idea in the following remark.

\begin{remark}
Given general points $p_1,\ldots, p_k\in \RR\subset \PP^1_{\trop}$, one obtains a toric degeneration of $\PP^1$, by using these points as vertices in a polyhedral decomposition of $\RR$ whose recession fan is the fan of $\PP^1$. A construction of Gubler produces the desired toric degeneration~\cite{Gub13} whose special fiber components are in bijection with the points $p_1,\ldots, p_k$. Each tropical cover then determines a (strata parametrizing) relative stable maps to the special fiber of this curve. After passing to a maximal degeneration, the constancy of the virtual fundamental class in flat families~\cite{Li01A} yields the result above, after some careful bookkeeping. An alternative modern perspective on this comes from logarithmic geometry -- rather than degenerating the target $\PP^1$, this information in the degeneration and the maximally degenerate map are carried by the logarithmic structure. See~\cite{AC11,Che10,GS13} for details. Both approaches have had fruitful generalizations to higher dimensional toric varieties. Every tropical curve $\Gamma$ in $\RR^n$ determines a toric degeneration of any toric variety whose one skeleton includes the recession fan of $\Gamma$. After passing to a maximal degeneration, one obtains similar correspondence results between tropical curve counts and algebraic counts, though not descendant invariants. This approach was pioneered by Nishinou and Siebert~\cite{NS06}. The approach of logarithmic Gromov--Witten theory is more flexible, generalizes immediately to higher dimensional target toric varieties, and allows one to study invariants with psi class insertions. See~\cite{Gro15,R15b}.

\end{remark}

\begin{example}\label{ex-d4}
Let us compute $\langle (1)^4 | \tau_{3}(pt), \tau_{3}(pt)| (1)^4 \rangle_{0,2}^{\PP^1,4,\bullet, trop}$.
We need source curves with two vertices, of valence $5$ and genus $0$, or valence $3$ and genus $1$. It is not possible to have a one-valent inner vertex: by the balancing condition the adjacent edge needs to be contracted leading to multiplicity zero.
The combinatorial types of the contributing tropical covers are depicted in Figure \ref{fig-exd4}. We depict genus one vertices as thickened dots. Vertices for which no genus is shown are assumed to be genus zero vertices. Expansion factors on edges which are not shown are assumed to be one. The number below each picture is the multiplicity of the corresponding cover. The automorphisms and expansion factors which contribute to the multiplicity are read off each picture. For the local vertex multiplicities, in Section \ref{subsec-g1mult} we show that  they are $1$ for genus zero vertices, and they are $\frac{1^2+1^2+2^2-1}{24}=\frac{5}{24}$  for the genus $1$ vertices.

\begin{figure}
\includegraphics{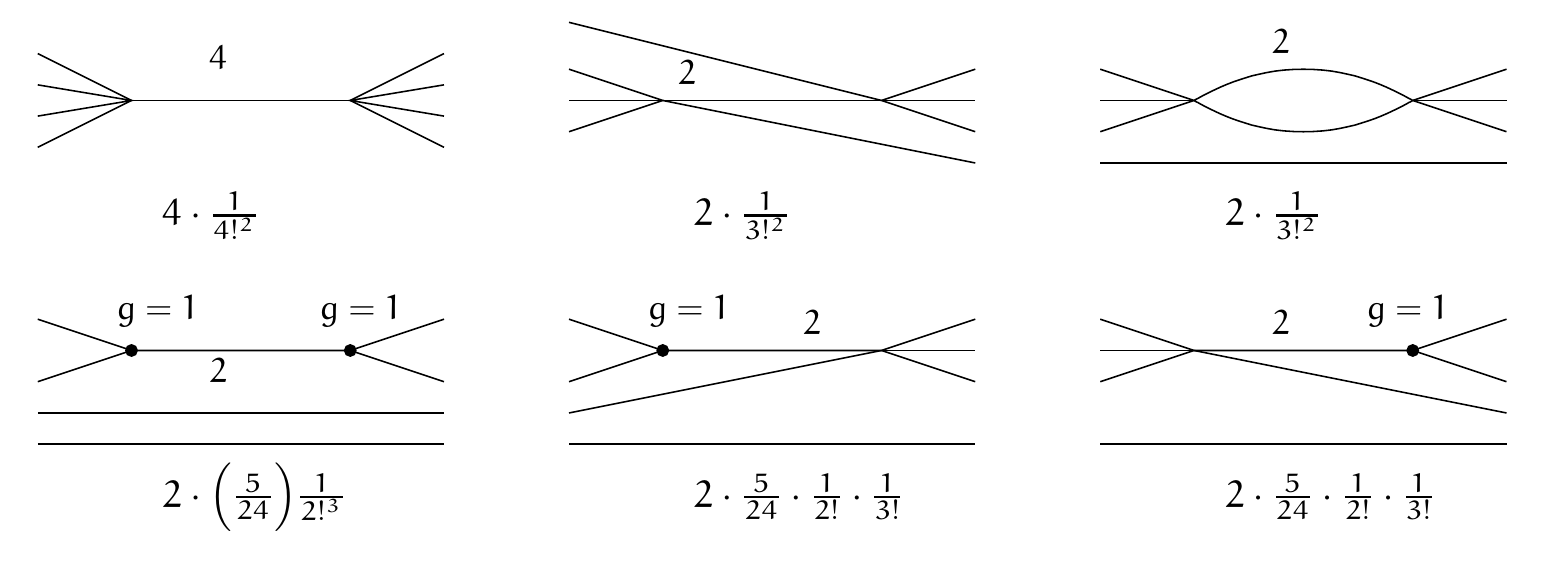}
\caption{The combinatorial types of tropical covers contributing to $\langle (1)^4 | \tau_{3}(pt), \tau_{3}(pt)| (1)^4 \rangle_{0,2}^{\PP^1,4,\bullet, trop}$.}\label{fig-exd4}
\end{figure}
 
\end{example}




\subsection{Correspondence theorem for general target curves.} 
Our correspondence theorem \ref{thm-corres} extends essentially without modification for a genus $1$ target curve. 

\begin{theorem}
Let $E$ be an elliptic curve. We have an equality 
$$\langle  \tau_{k_1}(pt)\ldots\tau_{k_n}(pt) \rangle ^{E,d}_{g,n}
= \sum_{\pi:\Gamma\to E^{\trop}} \frac{1}{| \Aut(\pi)|} \prod_V m_V \prod_e \omega(e).
$$
The sum is taken over all tropical covers $\Gamma \to E^{\trop}$ where $E^{\trop}$ is a cycle graph of any fixed edge length, and the covering and multiplicities are as in $ (A),(C),(D),(E)$ of Theorem \ref{thm-corres}.
\end{theorem}
\begin{proof}
The proof is analogous to the proof of  Theorem \ref{thm-corres}, after  one chooses a degeneration of the elliptic curve $E$  to a cycle of rational curves.
\end{proof}

In order to  extend the correspondence theorem to any target curve, we need local Hurwitz number as multiplicities for certain vertices. 

\begin{definition}\label{tropgwigentc}
Let $g,h$ be two non-negative integers, $Y$ an explicit, trivalent tropical curve of genus $h$,  $k_1, \ldots, k_n$ nonnegative integers such that $\sum k_i = 2g-2 -d(2h-2)$. We define
the {\it tropical descendant GW invariant}
$$\langle \tau_{k_1}(pt)\ldots\tau_{k_n}(pt) \rangle ^{Y,d, \trop}_{g,n}
= \sum_{\pi:\Gamma\to Y} \frac{1}{| \Aut(\pi)|} \prod_V m_V \prod_e \omega(e)
$$
where
\begin{enumerate}[(A)]
\item The map $\pi:\Gamma\to Y$ is a connected tropical cover.
\item There is a unique vertex of $\Gamma$ over each one of $n$ fixed points $p_1,\ldots, p_n$ on the edges of $Y$.
\item The unique vertex $v_i$ over $p_i$  has genus $g_i$, and valence $k_i+2-2g_i$.  If the star of $v_i$ has right (resp. left) hand side expansion factors given by $\mu_i$ (resp. $\nu_i$), the multiplicity of the vertex $v_i$ is defined to be
\begin{equation}
m_{v_i} = |Aut(\mu_i)||Aut(\nu_i)|\int_{\overline M_{g_i,1}(\PP^1,\mu_i,\nu_i)} \psi^{k_i} ev^\star(pt).
\end{equation}
\item Vertices of $\Gamma$ lying over the vertices of $Y$ have multiplicity $m_V= H(v)$, the local Hurwitz number.
\item The product of the expansion factors $\omega(e)$ runs over the set of all bounded edges of $\Gamma$.
\end{enumerate}
\end{definition}

The tropical descendant invariant just defined agrees with the corresponding classical invariant following the same reasoning as Theorem \ref{thm-corres}: after the base curve is degenerated to a collection of three pointed rational curves, the tropical graphs correspond to the dual graphs of the degenerate covers counted with the appropriate multiplicities. Over rational components where all three points are nodes, one obtains zero-dimensional moduli spaces of three-pointed relative stable maps to $\PP^1$, whose degrees are three-pointed Hurwitz numbers.  

\subsection{Genus zero and one vertex multiplicities}\label{subsec-g1mult} 

We conclude this section with a computation of the genus  zero and one vertex multiplicities. Such multiplicities are given, in generating function form, in \cite{OP06}; we find it valuable, in these simple cases, to provide a direct geometric derivation of these coefficients.

\begin{lemma} The genus zero vertex multiplicities are given by:
\begin{equation}
\label{gzeropix}
|Aut(\mu)||Aut(\nu)|\langle \mu| \tau_{\ell(\mu)+\ell(\nu)-2}(pt)|\nu \rangle ^{\PP^1,d}_{0,1} = 1.
\end{equation}
\end{lemma}

\begin{proof}
The first step of the proof consists in realizing the left hand side of \eqref{gzeropix} as an integral over a space of rubber stable maps \cite{MP06}. We adopt the convention that in rubber relative stable maps all relative points are marked, which allows us to absorb the automorphism factors:
\begin{equation}
\label{gzeropixt}
|Aut(\mu)||Aut(\nu)| \langle \mu| \tau_{\ell(\mu)+\ell(\nu)-2}(pt)|\nu \rangle ^{\PP^1,d}_{0,1} = \langle \mu| \tau_{\ell(\mu)+\ell(\nu)-2}|\nu \rangle ^{\sim \PP^1d,}_{0,1} 
\end{equation}

Once the relative points are marked, there is a stabilization morphism from the moduli space of rubber, relative stable maps to the moduli space of curves:
$$
st: \overline{\mathcal{M}}^\sim_{0,1} (\PP^1, \mu,\nu,d) \to \overline{M}_{0,1+\ell(\mu)+\ell(\nu) }.
$$ 
In genus zero, the map $st$ is easily seen to be birational. We need to compare the $\psi$ class on the moduli space of rubber relative stable maps with the pullback of the $\psi$ class on the moduli space of curves: there are a series of very similar comparison lemmas in the literature~\cite{bssz:psi, cm:cmb, Ion02} that can be easily adapted to this situation. The $\psi$ class on the moduli space of maps equals the pullback of the $\psi$ class on the moduli space of curves corrected by a positive multiple the divisor of maps where the special marked point lies on a rational component that becomes unstable after forgetting the map via $st$. Since in this case   $\psi_1$  is supported on a non-relative point, there is no such correction and the class  is pulled back from the moduli space of curves. We can therefore apply projection formula and compute:
$$
 \langle \mu| \tau_{\ell(\mu)+\ell(\nu)-2}|\nu \rangle ^{\sim \PP^1,d}_{0,1} 
 =  \int_{\overline{M}_{0,1+\ell(\mu)+\ell(\nu) }
} \psi^{\ell(\mu)+\ell(\nu)-2}_1 =1.
$$ 
\end{proof}

 \begin{lemma} The genus one vertex multiplicities are given by:
\begin{equation}
\label{gonepix}
|Aut(\mu)||Aut(\nu)|\langle \mu| \tau_{\ell(\mu)+\ell(\nu)}(pt)|\nu \rangle ^{\PP^1,d}_{1,1} = \frac{\sum \mu_i^2 +\sum \nu_i^2 -1}{24}.
\end{equation}
\end{lemma} 

\begin{proof}
As in the previous lemma,
we express the left hand side of \eqref{gonepix} as an integral over a space of rubber stable maps \cite{MP06}. Again, we adopt the convention that in rubber relative stable maps all relative points are marked, which allows us to absorb the automorphism factors,
\begin{equation}
\label{gonepixt}
|Aut(\mu)||Aut(\nu)| \langle \mu| \tau_{\ell(\mu)+\ell(\nu)}(pt)|\nu \rangle ^{\PP^1,d}_{1,1} = \langle \mu| \tau_{\ell(\mu)+\ell(\nu)}|\nu \rangle ^{\sim \PP^1,d}_{1,1} 
\end{equation}
and to have a well-defined stabilization morphism:
$$
st: \overline{\mathcal{M}}^\sim_{1,1} (\PP^1, \mu,\nu,d) \to \overline{M}_{1,1+\ell(\mu)+\ell(\nu) }.
$$ 
The  $\psi$ class is supported on a non-relative point, hence pulled back from the moduli space of curves;
we apply the projection formula and compute:
$$
 \langle \mu| \tau_{\ell(\mu)+\ell(\nu)}|\nu \rangle ^{\sim \PP^1,d}_{1,1}  = \int_{\overline{M}_{1,1+\ell(\mu)+\ell(\nu) }
} st_\ast([\overline{\mathcal{M}}^\sim_{1,1} (\PP^1, \mu,\nu,d) ]^{vir}) \psi^{\ell(\mu)+\ell(\nu)}_1.
$$ 
The pushforward of the virtual fundamental class in the moduli space of rubber relative stable maps is a tautological class, called {\it double ramification cycle} (DRC), that has received a great deal of recent attention~\cite{CMW,CJ16,GZ14,Hai11,MW13}. In \cite{Hai11}, Hain provided a formula for the restriction of the DRC to the locus of curves of compact type, which led Pixton to conjecture a formula for the DRC as a decorated sum of dual graphs, where the decorations depend on the combinatorics of the graphs. Pixton's conjecture is now a theorem of Janda, Pandharipande, Pixton and Zvonkine~\cite{JPPZ}. The reader is referred to the survey~\cite{Cav14} for details and references.

While in general Pixton's formula is rather involved, in genus one it can be described concisely. To make our computation a little more elegant, let us define a vector of integers
$$
\mathbf{x} = (x_1, \ldots, x_{1+\ell(\mu)+\ell(\nu)}) = (0, \mu_1, \ldots, \mu_{\ell(\mu)}, -\nu_1, \ldots, -\nu_{\ell(\nu)}),
$$
and we denote $DRC_1(\mathbf{x})$ the double ramification cycle we are after.

We have:
\begin{equation}\label{pixfor}
DRC_1(\mathbf{x}) =-\lambda_1+ \sum_{i=1}^{1+\ell(\mu)+\ell(\nu)} \frac{x_i^2\psi_i}{2} - \frac{1}{2}\left( \sum_{\mathcal{I} \subseteq [n+1]} (\sum_{i\in \mathcal{I}} x_i)^2 D(0,\mathcal{I})\right) ,
\end{equation}
where $D(0,\mathcal{I})$ denotes the divisor where the marks belonging to the subset $\mathcal{I}$ lie on a rational component, and the complementary marks on a genus one component. We now  compute  the degree of $DRC_1(\mathbf{x})\psi_1^{\ell(\mu)+\ell(\nu)}$.
Using string and dilaton equations (see for instance~\cite{Vak08}), one evaluates:
\begin{equation}
\label{easy}
\lambda_1 \psi_1^{\ell(\mu)+\ell(\nu)} = \frac{1}{24}, \ \ \ \ \ \ \ \ \ \ \ \  \psi_1\psi_i^{\ell(\mu)+\ell(\nu)} = \frac{\ell(\mu)+\ell(\nu)} {24} 
\end{equation}
By dimension reasons, $D(0,\mathcal{I})\psi_1^{\ell(\mu)+\ell(\nu)}=0$ unless $\mathcal{I}$ has precisely two elements, and in this case:
\begin{equation}
\label{alsoeasy}
 D(0,\mathcal{I})\psi_1^{\ell(\mu)+\ell(\nu)} = \frac{1}{24}
\end{equation}
Combining the evaluations \eqref{easy}, \eqref{alsoeasy} with  formula \eqref{pixfor}, we obtain:
$$
DRC_1(\mathbf{x})\psi_1^{\ell(\mu)+\ell(\nu)} = -\frac{1}{24}+ \sum_{i} \frac{(\ell(\mu)+\ell(\nu))x_i^2}{48} - \sum_{i\not=j} \frac{(x_i+x_j)^2}{48},
$$ 
 which is only some elementary algebraic simplifications away from the formula in Lemma \ref{gonepix}. We won't go into detail here, but remark that one may use the relation 
 $\sum_i x_i = 0$ to get rid of all the mixed terms of the form $x_i x_j$ and then  keep track of the remaining square terms.
 \end{proof} 
\section{Tropical GW/Hurwitz equivalence}\label{sec-gwh}

In this section we make a direct comparison between tropical Hurwitz numbers and tropical descendant invariants, which we call the \emph{tropical GW/Hurwitz equivalence}. 
Even though one may deduce this correspondence theorem by ``going around" Diagram \ref{square}  and using the algebraic GW/Hurwitz equivalence,  the direct connection shows once more that tropical geometry reveals a finer connection: the GW/Hurwitz equivalence can be broken down to a relation which works cover by cover, graph by graph, revealing tropical geometry again as a user-friendly graphical interface for GW-theory of curves. Note that the data of the completion coefficients (see \cite[Section 0.4]{OP06}) is used in the tropical GW/H-equivalence --- it should again be viewed as local input data which is needed on the tropical side.


\begin{lemma}[Local GW/Hurwitz correspondence]\label{loc-GWH}
Let $\Gamma$ be a graph consisting of one vertex of genus $g$ and $n$ ends labelled with weights $x_1, \ldots, x_{n}$ such that $\sum x_i=0$. Define $k = 2g+n-2$.  Denote by $\mu$ (resp. $\nu$) the partition of $d$ consisting of all positive entries of $\mathbf{x}$ (resp. the absolute value of all negative parts of $\mathbf{x}$). Let $\mathcal{S}$ denote the set of decorated, possibly disconnected, tropical covers of a tripod  $X \to T$ (see Figure~\ref{fig-tripod}) with a subset of ends marked, satisfying the following conditions:
\begin{figure}[h!]
\includegraphics{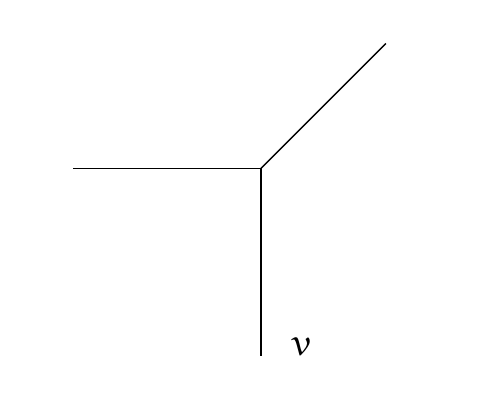}
\caption{The tripod $T$. We pick one of the ends of $T$ (in this picture marked by $v$) and refer to it as the \emph{vertical end} of $T$.}\label{fig-tripod}
\end{figure}

\begin{itemize}
\item Ignoring the genus labeling of vertices, $\Gamma=X/\sim $, where $\sim$ identifies all vertices of $X$ and contracts the ends covering the vertical end of the tripod $T$. We refer to the latter as the {\it vertical ends} of $X$.
\item All marked ends are vertical. 

		\item Any connected component of $X$ has at least one marked end.
		\item Every vertical end with expansion factor strictly larger than one is marked.\medskip

\noindent
	We denote by $\mu_X$ the partition whose parts are the expansion factors of the marked vertical ends
	\item  the genus $g_X$ of $X$ satisfies the following equation:
	$$
	2g_X + \ell(\mu_X)+k - |\mu_X|   = 2g.
	$$
\end{itemize}
Then the following equality holds:
\begin{equation}\label{localg}
\langle \mu | \tau_{k}(pt)| \nu \rangle_{g,1}^{\PP^1,d} = \sum_{X\in \mathcal{S}}\frac{ \rho_{k+1,\mu_X}}{k!} \prod_{v\in X}H(v),
\end{equation}
where $\rho_{k+1,\mu_X}$ is the completion coefficient (\ref{eq-compco}), and for all vertices $v$ of $X$, $H(v)$ denotes the local Hurwitz number (\ref{lochu}).
\end{lemma}

\begin{proof} Equation \eqref{localg} is a particular case of the classical GW/H correspondence. The base $\PP^1$ is degenerated to a nodal rational curve $T_1\cup T_2$ in such a way that the two relative conditions are on $T_1$, and the image of the marked point is on $T_2$. Given a relative stable map to such degenerate target,  the tropical curve $C$ is the dual graph of the inverse image of $T_1$, and the marked ends correspond to the nodes attaching to the connected component of the inverse image of $T_2$  supporting the marked point. Formula \eqref{localg} is then a combinatorial organization of the contributions of the degeneration formula (illustrated in an example in Figure \ref{fig:loc}).
\end{proof}

\begin{figure}[h!]
\includegraphics{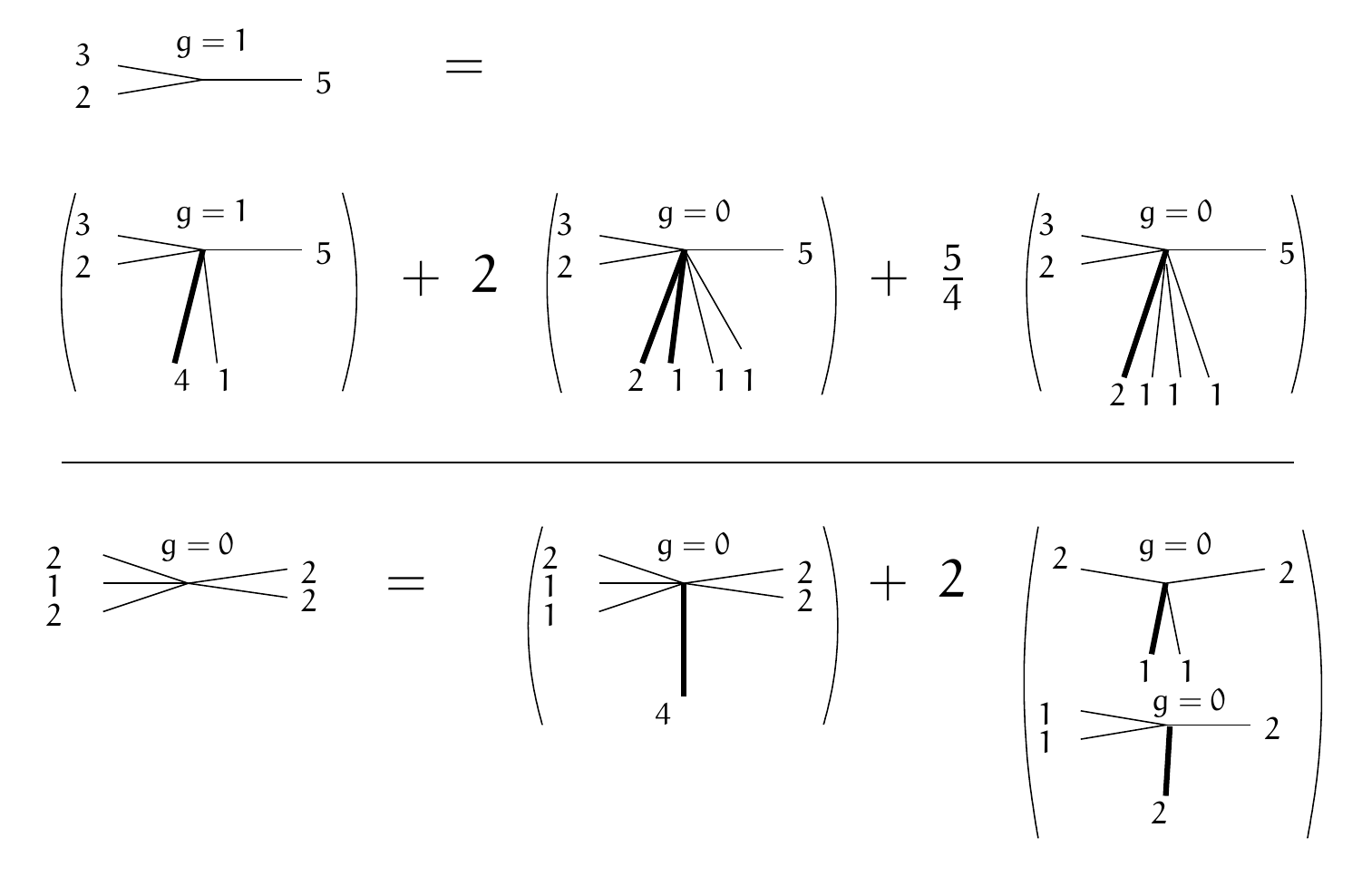}
\caption{Two examples of our local correspondence. In both cases we are looking at $k=3$, and the left hand side graph should be considered weighted by $3! m_v$, where $m_v$ the appropriate one pointed descendant invariant. In the first case, there are three possible ways of attaching vertical ends to the initial graph; note however than in the second and third graphs one has to lower the genus of the vertex by one. In the second example we can split our graph into two connected components (or a graph of genus $-1$), but each component must contain a marked vertical end.}\label{fig:loc}
\end{figure}

The general  tropical GW/H correspondence now consists in making the construction in Lemma \ref{loc-GWH} global. We wish to consider \eqref{localg} as an equation among the degrees of vectors consisting in formal linear combinations of appropriately decorated graphs, where $\deg(\sum m_i \Gamma_i):= \sum m_i$. 
In the following algorithm we establish a correspondence between the contribution of a graph that appears in the computation of a tropical descendant GW invariant, and the degree of linear combination of graphs that appear in the computation of appropriate Hurwitz numbers.
We define a {\it caterpillar tree} to be a  graph obtained by attaching vertical ends to a horizontal line (see Figure \ref{fig-caterpillar}). Arbitrary tropical Hurwitz numbers of $\PP^1$  (with $n$ ramification conditions) can be computed enumerating tropical covers of a caterpillar tree (with $n$ ends).

\begin{figure}
\includegraphics{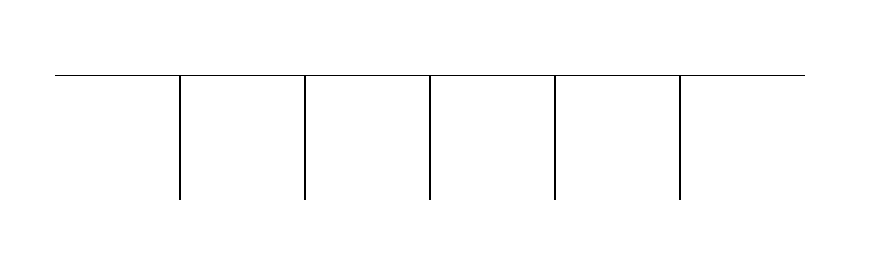}
\caption{A caterpillar tree.} \label{fig-caterpillar}
\end{figure}

\begin{algorithm} \label{tgwh}
Consider a tropical cover $f: \Gamma\to \PP^1_{\trop}$ contributing to the tropical descendant invariant $\langle \mu | \tau_{k_1}(pt),\ldots, \tau_{k_n}(pt)| \nu \rangle_{g,n}^{\PP^1, d,\bullet, trop}$ as in Definition \ref{tropgwi}.
Denote $Cont(\Gamma)$ the contribution of $\Gamma$ to the  GW invariant.  For every vertex $v$ of $\Gamma$, perform the following surgery on $Cont(\Gamma)\Gamma$:
\begin{itemize}
	\item replace $m_v star(v)$, which we think of as the left hand side of \eqref{localg}, with the linear combination of graphs $\sum_{X\in \mathcal{S}}  \left(\frac{ \rho_{k+1,\mu_X}}{k!} \prod_{v\in X}H(v)\right) X$, where all notation is as in Lemma \ref{loc-GWH}.
	\item for any point $x\not= v$ which is not a vertex $v$ but with $f(x)=f(v)$, $x$ belongs to some edge $e$ with some expansion factor $\omega(e)$. Attach $\omega(e)$ vertical ends to the point $x$, each with expansion factor $1$.
\end{itemize}
The result is linear combination of graphs that naturally cover a caterpillar tree $C$. Such graphs  appear in the tropical computation of $$H_g^{trop, \bullet}\left(\mu, \frac{\overline{k_1+1}}{k_1!},\ldots, \frac{\overline{k_n+1}}{k_n!}, \nu\right)$$
 with the same coefficients. The degree of such a linear combination is $Cont(\Gamma)$.
\end{algorithm}
The statement in the last sentence of the algorithm is a consequence of \eqref{localg}: the vertex multiplicities $m_v$ for tropical descendant invariants have been replaced by the appropriate local Hurwitz numbers and completion coefficients.

\begin{remark} \label{rem:tgwh}
The procedure in  Algorithm \ref{tgwh} has an inverse: start from a  tropical Hurwitz cover $f:X\to C$ of a caterpillar tree that contributes to $H_g^{trop, \bullet}\left(\mu, \frac{\overline{k_1+1}}{k_1!},\ldots, \frac{\overline{k_n+1}}{k_n!}, \nu\right)$; note that by \eqref{mutilde}, for each vertical end $l$ of the caterpillar tree we have also a marking  of a subset of vertical ends of $X$ mapping to $l$ with expansion factor $1$ (this marking is in fact the complement of the marking of the vertical ends in the algorithm). For $v$ a vertex of $C$, identify all vertices  in $f^{-1}(v)$ that are adjacent to some  non-marked edges. Shrink all vertical ends. This graph contributes to the computation of 
 $\langle \mu | \tau_{k_1}(pt),\ldots, \tau_{k_n}(pt)| \nu \rangle_{g,n}^{\PP^1, d,\bullet, trop}$.
\end{remark}

Algorithm $\ref{tgwh}$ and Remark \ref{rem:tgwh} together establish the tropical GW/H correspondence. But in fact, they give a more refined correspondence that matches the contributions of individual graphs contributing to GW invariants with the contributions of groups of graphs that contribute to the completed cycle Hurwitz numbers. As with many combinatorial algorithms, the general description appears much more complicated than the algorithm is: we conclude the section by illustrating the algorithm with an example.

\begin{example}
Let us apply Algorithm \ref{tgwh} to the two left covers in the top row of Figure \ref{fig-exd4}, contributing  $4\cdot \frac{1}{4!^2}$ to
$\langle (1)^4 | \tau_{3}(pt), \tau_{3}(pt)| (1)^4 \rangle_{0,2}^{ \PP^1, 4, \bullet,trop}$.
For each vertex of the  top left cover, the local algorithm returns only one possible local modification, consisting in attaching a  marked vertical end with expansion factor $4$. The resulting cover of a caterpillar tree  is depicted in Figure \ref{fig-alg1}. 
 The multiplicity of this Hurwitz cover is $4\cdot \frac{1}{4!^2}$ since all local Hurwitz numbers for vertices are one, and so is the completion coefficient. This corresponds to the contribution of the initial graph to the descendant invariant.

\begin{figure}
\includegraphics{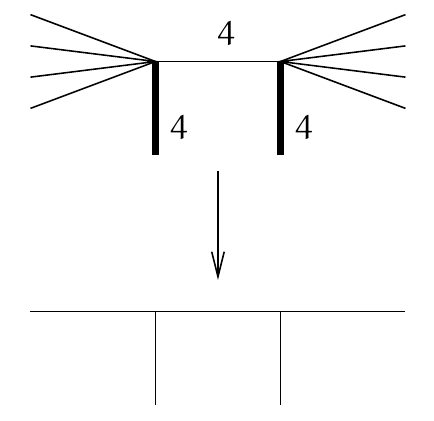}
\caption{The Hurwitz cover we obtain from the top left cover of Figure \ref{fig-exd4} following Algorithm \ref{tgwh}. Marked vertical ends are drawn with a thick line.}\label{fig-alg1}
\end{figure}


We turn  to the top middle cover of Figure \ref{fig-exd4}, which contributes  $2 \frac{1}{(3!)^2}$ to the descendant invariant
$\langle (1)^4 | \tau_{3}(pt), \tau_{3}(pt)| (1)^4 \rangle_{0, 2}^{ \PP^1,4, \bullet, trop}$.
Again, the local algorithm offers only one possibility for each vertex. The result of the surgery is in Figure \ref{fig-alg2}. The Hurwitz 
cover is weighted, for each vertex of the caterpillar tree, by the completion coefficient $\rho_{(2,1),4}=\frac{1}{3!}\cdot 2$; all local Hurwitz numbers are $1$. We have three compact edges with expansion factors $2,1, 1$ and two pairs of balanced forks, giving the cover an automorphism group of order $4$. All together the contribution of this cover is 
$2 \frac{1}{(3!)^2}$, which coincides with the contribution of the original graph to the descendant invariant.



\begin{figure}
\includegraphics{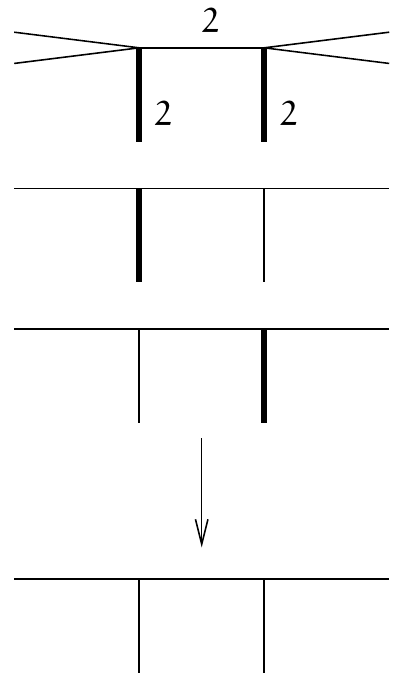}\caption{The Hurwitz cover we obtain from the top middle cover of Figure \ref{fig-exd4} following Algorithm \ref{tgwh}. Marked vertical edges are drawn with a thick line.}\label{fig-alg2}
\end{figure}
\end{example}

 \section{Fock spaces and Feynmann diagrams}
\label{fock}
 
In this section, we introduce bosonic Fock space, and explain how to express Gromov-Witten invariants as matrix elements on bosonic Fock space.  These matrix elements have expansions as Feynman diagrams, which are closely related to the tropical Gromov-Witten invariants we have discussed, as captured in the slogan 

\begin{center}
{\it ``Bosonification is Tropicalization''.}
\end{center}

 \subsection{Preliminaries on bosonic Fock}


The bosonic Heisenberg algebra $\mathcal{H}$ is the Lie algebra with basis $a_n, n\in \mathbb{Z}$ satisfying commutator relations
$$ [a_n, a_m]=(n\cdot \delta_{n,-m})a_0$$
where $\delta_{n,-m}$ is the Kronecker symbol. Note that $a_0$ is central, and is often denoted $c$ in other sources. 

The bosonic Fock space $F$ is a particular representation of $\mathcal{H}$.  It is generated by a single ``vacuum vector'' $v_\emptyset$.  The positive generators annihilate $v_\emptyset$: $a_n\cdot v_\emptyset=0$ for $n>0$; $a_0$ acts as the identity, and the negative operators act freely.  That is, $F$ has a basis $b_\mu$ indexed by partitions, where 
$$b_\mu=a_{-\mu_1}\dots a_{-\mu_m}\cdot v_{\emptyset}$$

In more familiar terms, the bosonic Fock space is isomorphic to a polynomial ring in infinitely many variables, $p_n, n>0$. The vacuum vector $v_\emptyset$ is the element $1\in \mathbb{C}[p_1,p_2,\dots]$, and the operators $a_n$ act as follows:
$$
a_n=\left\{\begin{array}{ll}  p_{-n}\cdot & n<0 \\ 
n\frac{\partial}{\partial p_n} & n>0 \end{array}\right.
$$
The vector $b_\mu$ is mapped to the monomial $p_{\mu_1}\cdots p_{\mu_m}$ under this isomorphism.

 We define an inner product on $F$ by declaring $\langle v_\emptyset | v_\emptyset \rangle=1$ and $a_n$ to be the adjoint of $a_{-n}$. Thus we have 
$$ \langle b_{\mu}| b_{\nu} \rangle = \mathfrak{z}(\mu) \delta_{\mu,\nu}.$$
where
$$\mathfrak{z}(\mu)=|\Aut(\mu)|\cdot\prod\mu_i$$
is the size of the centralizer of an element of cycle type $\mu$ in $S_{|\mu|}$.

Note that being a representation of a Lie algebra is equivalent to being equivalent to a representation of the universal enveloping algebra of that Lie algebra.  All this means in our case is that products of elements of $\mathcal{H}$ act on $F$, not just commutators of these elements.

Following standard conventions, we write 
$\langle \alpha|A|\beta\rangle$ for $\langle \alpha|A\beta\rangle$ for $\alpha, \beta \in F$ and an operator $A$ that is a product of elements $\mathcal{H}$. Such expressions are referred to as \emph{matrix elements}. We also write $\langle A \rangle$ for $\langle v_\emptyset |A|v_\emptyset \rangle$, such a value is called a \emph{vacuum expectation}.

Our goal now is to write relative Gromov-Witten invariants as  bosonic matrix elements.  This is largely a matter of rephrasing certain results from \cite{OP06}.  

\begin{remark}
Okounkov-Pandharipande \cite{OP06} express these Gromov-Witten invariants as matrix elements on the \emph{fermionic Fock space}, also known as the \textit{infinite wedge}.  Fermionic Fock space is closely related to bosonic Fock space by the boson-fermion correspondence, and it is possible to use the general method of ``bosonification'' to go directly from Okounkov-Pandharipande's expression to a bosonic one -- this is the approach taken by \cite{Li11}.  We take a different approach, bypassing any need for further discussion of fermionic Fock space.   For more on Fock spaces and the boson-fermion correspondence, we recommend \cite{KR}. 
\end{remark}

Before tackling Gromov-Witten theory, we express double Hurwitz numbers $H^\bullet_{g\rightarrow 0}(\mu,\nu)$ as a matrix element on bosonic Fock space as a warm-up.

\subsection{Hurwitz theory on the bosonic Fock space}

The normal ordering product sorts a product of operators $a_{k_i}$ so that the $k_i$ are decreasing.  

\begin{definition}
Let $k_i, 1\leq i \leq n$ be any integers, and let $\ell_i, 1\leq i\leq n$ be the same multiset of integers sorted in decreasing order, $\ell_1\geq \ell_2\geq \cdots \geq \ell_n$.  The \emph{normal ordered product} of the $a_{k_i}$, denoted $:\prod_{i=1}^n a_k:$, is defined by

$$:\prod_{i=1}^n a_{k_i} :\;\; = \prod_{i=1}^n a_{\ell_i}$$
\end{definition}
An important example of normal ordering occurs in the cut-join operator.

\begin{definition} The \emph{cut-join} operator is defined by:
\begin{equation}
\label{caj}\cutjoin =\sum_{i+j+k=0}\frac{1}{6} :a_i a_j a_k:\;\; = \frac{1}{2} \sum_{k\geq 0} \sum_{\substack{0\leq i\leq j \\i+j=k}} a_{-j}a_{-i}a_k+a_{-k}a_{i}a_{j}
\end{equation}
\end{definition}

The cut-join operator encodes the effect of multiplication by a transposition in the group center of the group algebra of $S_n$.  Recall that $C_\mu$ is the sum of all elements in the conjugacy class identified by the partition $\mu$, and denote by $T$ the sum of all transpositions.  If
$$T\cdot C_\mu=\sum_{\nu\vdash|\mu| } k_\mu^\nu C_\nu$$
then we have
$$\cutjoin b_\mu=\sum_{\nu\vdash |\mu|} k_\mu^\nu b_\nu$$
Intuitively, the meaning of the cut-join operator is that multiplying by a transposition either cuts a cycle of a permutation into two cycles, or joins two cycles into one, which are the content of the first and second terms in the last part of Equation \eqref{caj}.

It follows that double Hurwitz numbers can be written as matrix elements of the appropriate power of the cut-join operator:

$$H^\bullet_{g\to 0}(\mu,\nu)=\frac{1}{\mathfrak{z}(\mu)}\frac{1}{\mathfrak{z}(\nu)}\langle b_\mu |\cutjoin^r |b_\nu\rangle$$

\subsection{Gromov-Witten theory in the bosonic Fock space}

Gromov-Witten invariants may be expressed as vacuum expectations in a similar manner.  Each insertion of $\tau_k(pt)$ corresponds to an operator similar to, but slightly more complicated than, the cut-join operator.

The starting place for our bosonic expression for Gromov-Witten theory is the following result:

\begin{definition}
Define $\mathcal{S}(z)=\displaystyle{\frac{\mathrm{sinh}(z/2)}{z/2}}$.
\end{definition}

\begin{theorem}[{~\cite[Theorem 2]{OP06}}]\label{thm-onepointseries}
The one-point series is
$$ \sum_g \langle \mu|\tau_{2g-2+\ell(\mu)+\ell(\nu)}(pt)|\nu \rangle_{g,1}^{\mathbb{P}^1,|\mu|,\bullet} \cdot z^{2g}= \frac{1}{|\Aut(\mu)||\Aut(\nu)|}\frac{\prod \mathcal{S}(\mu_iz)\prod \mathcal{S}(\nu_iz)}{\mathcal{S}(z)}.$$
\end{theorem}

Motivated by this result, we make the following definition:

 \begin{definition}\label{def-operator}
 For $k\in\mathbb{Z}, k>0$, define the operator
 $$ M_k=\sum_{g\geq 0} \sum_{{\bf x}} \langle {\bf x}^+|\tau_k(pt) |{\bf x}^- \rangle_{g,1}^{\PP^1,|{\bf x}^+|, \bullet} \cdot u^{l-1+g}\cdot a_{x_1}\cdot \ldots \cdot a_{x_{k+2-2g}}\in \mathcal{H}[u],$$
 where the second sum goes over all ${\bf x} \in \ZZ^{k+2-2g}$ satisfying $$x_1\leq \ldots \leq x_l <0< x_{l+1}\leq \ldots \leq x_{k+2-2g}$$ and $$\sum_{i=1}^{k+2-2g}x_i=0.$$
 \end{definition}
 
\begin{theorem}\label{thm-GWIFock} The disconnected stationary relative descendant Gromov-Witten invariant of $\PP^1$ can be expressed as a matrix element in the Fock space:
\begin{equation}\label{eq-GWIFockunmarkedends}\langle \mu|\tau_{k_1}(pt)\ldots \tau_{k_n}(pt)|\nu\rangle_{g,n}^{\PP^1,|\mu|,\bullet}= \langle b_\mu|\mbox{Coeff}_{u^{g+\ell({\mu})-1}}(M_{k_1}\cdot \ldots \cdot M_{k_n})|b_\nu\rangle.\end{equation}
\end{theorem} 
 
 Notice first that for any operator $M\in \mathcal{H}$
 \begin{align*}
  \langle b_\mu|M|b_\nu\rangle &= \frac{1}{|\Aut(\mu)|\cdot |\Aut(\nu)|}\cdot \langle v_\emptyset | a_{\mu_1}\cdot \ldots\cdot a_{\mu_{\ell(\mu)}}\cdot M\cdot a_{-\nu_1}\cdot \ldots\cdot a_{-\nu_{\ell(\nu)}} |v_\emptyset\rangle\\ &=  \frac{1}{|\Aut(\mu)|\cdot |\Aut(\nu)|}\cdot \langle  a_{\mu_1}\cdot \ldots\cdot a_{\mu_{\ell(\mu)}}\cdot M\cdot a_{-\nu_1}\cdot \ldots\cdot a_{-\nu_{\ell(\nu)}} \rangle.
  \end{align*}
Thus Equation (\ref{eq-GWIFockunmarkedends}) is equivalent to the following equality involving a disconnected descendant Gromov-Witten invariant and a vacuum expectation as above:

\begin{align}
\label{eq-GWIFockmarked}
\begin{split}
& |\Aut(\mu)|\cdot |\Aut(\nu)| \cdot \langle \mu|\tau_{k_1}(pt)\ldots \tau_{k_n}(pt)|\nu\rangle_{g,n}^{\PP^1,|\mu|,\bullet}\\&= \langle  a_{\mu_1}\cdot \ldots\cdot a_{\mu_{\ell(\mu)}}\cdot\mbox{Coeff}_{u^{g+\ell({\mu})-1}}(M_{k_1}\cdot \ldots \cdot M_{k_n}) \cdot a_{-\nu_1}\cdot \ldots\cdot a_{-\nu_{\ell(\nu)}} \rangle.
 \end{split}
\end{align}

\subsection{Feynman diagrams as tropical curves}

Vacuum expectations like the one appearing in Equation (\ref{eq-GWIFockmarked}) can be computed as the weighted sum over \emph{Feynman graphs} associated to a monomial contributing to the value. This can be viewed as a variant of Wick's theorem \cite{Wic50}. A similar version can be found in~\cite[Proposition 5.2]{BG14}. In our situation, the Feynman graphs in question are essentially tropical covers and Theorem~\ref{thm-GWIFock} follows by combining Theorem~\ref{thm-corres} and Wick's Theorem.

\begin{definition}\label{def-Feynmangraph}
Let $P=m_0\cdot m_1\cdot \ldots \cdot m_n\cdot m_{n+1}$ where each $m_i$ is a normally ordered product of some $a_s$, and such that $m_0$ contains only factors $a_s$ with $s>0$ and $m_{n+1}$ contains only factors $a_s$ with $s<0$.

We associate graphs to $P$ called \emph{Feynman graphs} for $P$. 

Each Feynman graph of $P$ will have $n$ vertices $v_1,\ldots, v_n$, which are linearly ordered. 

The terms $a_s$ appearing in the product $m_i$ determine the local structure of the graph at $v_i$. For each factor $a_s$ appearing in $m_i$, we draw an edge germ of weight $|s|$ which is directed to the left if $s<0$ and to the right if $s>0$. 

The monomials $m_0$ and $m_{n+1}$ determine the unbounded edges of the Feynman graph.   For $m_0$, we draw (weighted) germs of ends left of $v_0$ (necessarily directed to the right). For $m_{n+1}$, we draw (weighted) germs of ends right of $v_n$ (necessarily directed to the left). All edge germs are marked to make them distinguishable. The resulting picture is called the Feynman fragment associated to $P$.

A Feynman graph for $P$ is any weighted graph on $v_1,\ldots,v_n$  completing the Feynman fragment, i.e.\ any weighted graph in which the marked weighted edge germs are completed to whole edges. Two edge germs can only be completed to an edge if 
\begin{itemize}
\item one is directed to the right and one to the left, and the vertex adjacent to the germ directed to the right is smaller than the one adjacent to the germ directed to the left, and
\item the two edge germs have the same weight.
\end{itemize}

\end{definition}

 \begin{example}\label{ex-Feynman}
Let $$P=(a_{1}^4)\cdot (a_{-1}^3 \cdot a_2 \cdot a_1) \cdot (a_{-2}\cdot a_{-1}\cdot a_1^3)\cdot (a_{-1}^4)$$
be a product satisfying the requirements of Definition \ref{def-Feynmangraph}. Notice that this product appears in the expansion of $\langle a_1^4\cdot M_3\cdot M_3\cdot a_{-1}^4\rangle$. Figure \ref{fig-edgegerms} shows the corresponding Feynman fragment for this product. The small numbers are markings for the edge germs, the bigger ones weights. Edge weights which are one are not indicated in the picture. Notice that the tropical covers depicted in Figure \ref{fig-exd4} on the top middle, and on the top right, come from Feynman diagrams completing these edge germs, after forgetting the marking of edge germs.
\begin{figure}
\includegraphics{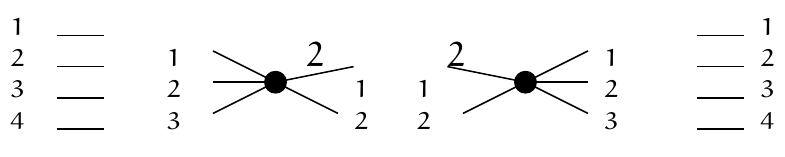}
\caption{The Feynman fragment corresponding to the product $P$ of Example \ref{ex-Feynman}. The small numbers are markings for the edge germs, the bigger ones weights. Edge weights which are one are not indicated in the picture.}\label{fig-edgegerms}
\end{figure}
 
 \end{example}
 
 \begin{proposition}[Wick's Theorem, see{~\cite[Proposition 5.2]{BG14}}]\label{prop-wick}
The vacuum expectation $\langle P\rangle$ for a product $P$ as in Definition \ref{def-Feynmangraph} equals the weighted sum of all Feynman graphs for $P$, where each Feynman graph is weighted by the product of weights of all edges (interior edges and ends).
\end{proposition}

\begin{proof}
The proof goes by iteratively using the commutation relations for the Heisenberg algebra to move the operators $a_i$ with $i>0$ to the right.

Consider the rightmost $a_i$ with $i>0$ in the product $P$. If this is the rightmost term in the whole product, we obtain zero since by definition $a_i\cdot v_\emptyset=0$ for $i>0$. 

If this $a_i$ is not the rightmost term, then there is an operator $a_j$ appearing immediately after it, and we make the substitution
$$a_ia_j=a_ja_i+[a_i,a_j]$$

The first term on the right hand side has just made $a_j$ jump past the $a_i$.  When $j\neq -i$, this is the only contribution we have.

If $a_i$ is immediately followed by $a_{-i}$, however, we get a second term: $[a_i,a_{-i}]=i$.  In this term, the two operators $a_i$ and $a_{-i}$ have cancelled each other and produced the scalar $i$.  

We see that this replacement produces two terms, both of which are simpler than the original product $P$ -- the first one has an $a_i$ with $i>0$ further to the right, and the second term has less $a_i$ in the product.  Thus, we can iteratively repeat the substitution on each of the resulting terms, and eventually every term will either vanish because the rightmost $a_i$ has $i>0$, or we will get a term that consists solely of a scalar.

To each scalar term produced, we build a Feynman graph as follows: every time we cancel an $a_i$ and an $a_{-i}$ appearing right of $a_i$ in $P$, we add to the Feynman fragment of $P$ by drawing an edge connecting the germ corresponding to $a_i$ and the germ corresponding to $a_{-i}$.

Obviously, we draw all Feynman graphs completing the Feynman fragment for $P$ in this way. Each Feynman graph corresponds to a way to group the factors of $P$ in pairs $\{a_i,a_{-i}\}$ corresponding to edges completing the corresponding marked edge germs. Each such pair produces a contribution of $i$ because of the commutator relations, so altogether each Feynman graph should be counted with its weight to produce $\langle P \rangle$.
\end{proof}

\begin{proof}[Proof of Theorem \ref{thm-GWIFock}:]
As we have already seen, Theorem \ref{thm-GWIFock} is equivalent to Equation (\ref{eq-GWIFockmarked}). By Theorem \ref{thm-corres}, the left hand side equals the count of tropical covers satisfying properties (A)-(D) (up to the factor of $|\Aut(\mu)|\cdot |\Aut(\nu)|$, which can be viewed as numbers of ways to mark the ends).  By Proposition \ref{prop-wick}, each term contributing to the right hand side can be expressed in terms of a weighted count of suitable Feynman diagrams.  We show that the tropical covers contributing to the left hand side are essentially equal to the Feynman graphs contributing to the right (up to marking of edge germs), and that they are counted with the same weight on both sides.

Expand the right hand side as a sum of vacuum expectations, where each summand is of the form $w_P \cdot P$ for some scalar $w_P$ and $P=m_0\cdot \ldots\cdot m_{n+1}$ a product of the form described in Definition \ref{def-Feynmangraph}. For each summand, $m_0=\prod_{i}a_{\mu_i}$ and $m_{n+1}=\prod_{i}a_{-\nu_i}$. A factor $m_i$ for $i=1,\ldots,n$ comes from a summand of $M_{k_i}$, i.e.\ is of the form 
$$\langle {\bf x}^+|\tau_k(pt) |{\bf x}^- \rangle_{g_i,1}^{\PP^1,|{\bf x}^+|,\bullet} \cdot a_{x_1}\cdot \ldots \cdot a_{x_{k_i+2-g_i}}$$ where $g_i$ in the power of $u$.

  Enrich the Feynman fragment for $P$ by adding the genus $g_i$ to the vertex $i$ (imposed by the power of $u$). Any Feynman diagram completing this Feynman fragment can be interpreted as a tropical cover contributing to the tropical analogue of the left hand side. The balancing condition is satisfied at every vertex since we require $\sum x_i=0$ for the product above in $M_k$. The valence of vertex $i$ is $k_i+2-g_i$ as required. The degree of the tropical cover is given by $\mu$ and $\nu$. 

To see that the tropical cover is of the right genus, observe that the variable $u$ records the genus.  Built a Feynman graph from left to right, starting with the left ends, and adding in vertex after vertex from $1$ to $n$, taking the change in genus into account at each step. The genus of the graph consisting of $ \ell(\mu)$ left ends (at first disconnected) has genus $-\ell(\mu)+1$. For the vertex $i$ of local genus $g_i$, by definition of the operator $M_k$, we get a contribution of $u^{g_i+l_i-1}$, where $l_i$ denotes the number of incoming edges. Since $h$ incoming edges potentially close up $h-1$ cycles, the vertex $i$ increases the genus by $g_i+l_i-1$. By taking the $u^{g+\ell(\mu)-1}$ coefficient in total, we thus obtain tropical covers of genus $g$.

Each Feynman diagram for $P$ can, after forgetting the marking of edge germs, be viewed as a tropical cover contributing to the left hand side of \eqref{eq-GWIFockunmarkedends}, and vice versa, each tropical cover gives a Feynman graph.

We group Feynman graphs into equivalence classes if they only differ by the marking of their edge germs. These equivalence classes then bijectively correspond to tropical covers. The number of elements in an equivalence class for a given tropical cover $\pi:\Gamma\rightarrow \mathbb{R}$ equals
$$ \frac{1}{|\Aut(\pi)|}\cdot \prod_i |\Aut(\mu_i)|\cdot |\Aut(\nu_i)|,$$ where the product goes over all vertices $i=1,\ldots,n$ of $\Gamma$, $\mu_i$ denotes the partition of weights of incoming edges of $i$, and $\nu_i$ denotes the partition of weights of outgoing edges. 

It remains to show that an equivalence class of Feynman diagrams (resp.\ a tropical cover)  contributes to the left hand side (resp.\ right hand side) of \eqref{eq-GWIFockunmarkedends} with the same multiplicity. For the right hand side, note that a Feynman diagram contributes with its weight times $\prod_{i=1}^n \langle \mu_i| \tau_{k_i}(pt)|\nu_i\rangle_{g_i,1}^{\PP^1,|\mu_i|, \bullet}$, since the latter is the coefficient $w_P$ of the product $P$ in the expansion of the product of operators. 
Multiplying with the number of elements in an equivalence class of Feynman graphs, and with the nultiplicity of the Feynman graph, we obtain
$\frac{1}{|\Aut(\pi)|}\cdot \prod_{i=1}^n m_i\cdot \prod_e w(e)$,
where $m_i$ denotes the vertex multiplicity as defined in \eqref{localmv}. The latter equals the multiplicity of the tropical cover by Definition \ref{tropgwi}. 
  
\end{proof}

The expression of stationary relative Gromov-Witten invariants in terms of matrix elements in Theorem \ref{thm-GWIFock} is not completely satisfactory, since the operators $M_k$ used on the right hand side have one point Gromov-Witten invariants as coefficients themselves.
But we know by Theorem \ref{thm-onepointseries} that the one point Gromov-Witten invariants can be combined into a generating function which then can be given by a closed formula.

We can thus reformulate the operators used on the right hand side in Equation \ref{eq-GWIFockunmarkedends}.

\begin{definition}
Define $$\hat{a}_n=\begin{cases}u a_n &\mbox{ if } n<0\\ a_n &\mbox{ if } n>0 \end{cases}.$$

Let
$$ A_k^g= \frac{1}{(k+2-2g)!}\mbox{Coeff}_{w^0}\left(: \left(\sum_{x\in \ZZ}\mathcal{S}(xz)\hat{a}_x w^x \right)^{k+2-2g}:\right).$$

\end{definition}
\begin{proposition}
The operator series $M_k$ defined in Definition \ref{def-operator} satisfies
\begin{equation}
\label{eqtobeproved}
M_k=\mbox{Coeff}_{z^0}\left(\frac{1}{\mathcal{S}(z)}\cdot \sum_g z^{-2g} u^{g-1} A_k^g \right).
\end{equation}

\end{proposition}
 
\begin{proof}
Expand the expression defining $A_k^g$. Every summand is a product with $k+2-2g$ factors $\mathcal{S}(x_1z)\cdot \ldots \cdot \mathcal{S}(x_{k+2-2g}z) \hat{a}_{x_1}\cdot \ldots \cdot \hat{a}_{x_{k+2-2g}}$. Taking the $w^0$ coefficient in the definition of $A_k^g$ takes care of the balancing condition $\sum x_i=0$. Taking the normal ordering product produces products which are ordered as required, with the negative indices coming first and the positive indices later. The convention for $\hat{a}_s$ guarantees that we obtain a factor of $u^l$, where $l$ denotes the number of negative indices. 

A summand $\mathcal{S}(x_1z)\cdot \ldots \cdot \mathcal{S}(x_{k+2-2g}z) \hat{a}_{x_1}\cdot \ldots \cdot \hat{a}_{x_{k+2-2g}}$ appears $\frac{(k+2-2g)!}{|\Aut({\bf}x)|}$ times in the expansion. The factorial $(k+2-2g)!$ cancels with the prefactor in the definition of $A_k^g$.

Now expand the right hand side of \eqref{eqtobeproved}. The factors $\mathcal{S}(x_iz)$ and $\frac{1}{|\Aut({\bf}x)|}$ from the above discussion together with the factor $\frac{1}{\mathcal{S}(z)}$ in the definition of $M_k$ give the one-point series. Taking the $z^0$ coefficient after multiplying each term with $z^{-2g}$ sorts out the right one-point invariant appearing as a coefficient in $M_k$.
\end{proof} 
 
\bibliographystyle{siam}
\bibliography{GWGUI}

\end{document}